%% file: main.tex
\documentclass{article} 
\usepackage{arxiv_preprint,times}

\input{math_commands.tex}

\usepackage{hyperref}
\usepackage{url}

\usepackage{algorithm}
\usepackage{algpseudocode}
\usepackage{graphicx}
\usepackage{amssymb} 
\usepackage{amsmath}
\usepackage{amsthm}
\usepackage{booktabs}
\usepackage{longtable}

\newtheorem{proposition}{Proposition}

\newtheorem{corollary}{Corollary}

\usepackage{xcolor}
\usepackage{tikz}

\usetikzlibrary{
    arrows.meta,
    positioning,
    calc,
    shapes.geometric
}

\definecolor{arcblue}{RGB}{1,115,178}
\definecolor{arcorange}{RGB}{222,143,5}
\definecolor{arcgreen}{RGB}{2,158,115}
\definecolor{arcred}{RGB}{213,94,0}
\definecolor{arcpurple}{RGB}{204,120,188}

\newcommand{\arccandidate}{%
    \tikz[baseline=-0.55ex]{%
        \fill[gray!65] (0,0) circle (1.05mm);
    }%
}

\newcommand{\arcvalidnonhit}{%
    \tikz[baseline=-0.55ex]{%
        \draw[gray!65,line width=0.45pt]
            (0,0) circle (1.05mm);
    }%
}

\newcommand{\arcinvalid}{%
    \textcolor{gray!65}{\(\times\)}%
}

\newcommand{\arcstar}[1]{%
    \tikz[baseline=-0.55ex]{%
        \node[
            star,
            star points=5,
            star point ratio=2.25,
            fill=#1,
            draw=white,
            line width=0.25pt,
            minimum size=3.4mm,
            inner sep=0pt
        ] {};
    }%
}

\newcommand{\arcscenarios}{%
    \tikz[baseline=-0.45ex]{%
        \draw[
            fill=arcblue!12,
            draw=arcblue,
            line width=0.4pt
        ] (0,0) ellipse (1.4mm and 0.8mm);
        \draw[
            fill=arcorange!12,
            draw=arcorange,
            line width=0.4pt
        ] (2.0mm,0) ellipse (1.4mm and 0.8mm);
        \draw[
            fill=arcpurple!12,
            draw=arcpurple,
            line width=0.4pt
        ] (4.0mm,0) ellipse (1.4mm and 0.8mm);
    }%
}

\newcommand{\arcrisk}{%
    \tikz[baseline=-0.45ex]{%
        \draw[
            fill=arcred!10,
            draw=arcred,
            line width=0.4pt
        ] (0,0) ellipse (1.8mm and 0.9mm);
        \draw[arcred,line width=0.3pt]
            (-1.2mm,-0.55mm) -- (-0.55mm,0.55mm);
        \draw[arcred,line width=0.3pt]
            (-0.35mm,-0.75mm) -- (0.45mm,0.65mm);
        \draw[arcred,line width=0.3pt]
            (0.55mm,-0.55mm) -- (1.15mm,0.45mm);
    }%
}

\newcommand{\archit}{%
    \textcolor{arcgreen}{\(\checkmark\)}%
}

\title{Anchored Scenario Coverage for Failure-Aware First-Hit Batch Inverse Design}

\author{Chuhan Yang, Chenxi Wang, Linhan Wu \& Yuyang Liu \\
DeepVerse\\
\texttt{\{chuhan,chenxi,linhan,yuyang\}@deepverse.tech} 
}

\begin{document}

\maketitle

\begin{abstract}
Early discovery of at least one valid design satisfying a target requirement is a central objective in failure-prone closed-loop inverse design. A natural batch baseline ranks candidates by a product-form marginal valid-hit score,
but selecting the highest-ranked candidates independently can produce redundant recommendations under predictive uncertainty and waste the experiment budget. We introduce ARC-SC(Anchored Risk-Constrained Scenario Coverage), a batch acquisition method that preserves strong marginal candidates as anchors and allocates the remaining batch positions by maximizing complementary coverage over predictive target scenarios under a risk-support constraint. In frozen-oracle closed-loop simulations on superconductivity and JARVIS materials-property benchmarks, ARC-SC yields
a statistically supported improvement in first-hit discovery and remains competitive with directionally favorable first-hit performance on more challenging design space. These results establish ARC-SC as a POF-anchored, scenario-aware batch strategy for improving early valid-target discovery under structured experimental failure.

\end{abstract}

\section{Introduction}
\label{sec:introduction}

Closed-loop inverse design refers to an iterative discovery paradigm in which candidate solutions are sequentially proposed by computational models, evaluated through experiments or simulations, and used to update subsequent design decisions\citep{xue2016accelerated,kusne2020fly,wu2026matformbench}. In many applications, however, the practical objective is not necessarily to identify the global optimum or maximize the terminal best-observed value. The immediate goal is to discover, using as few experimental rounds as possible, at least one valid candidate satisfying a predefined target requirement. This first-hit objective is particularly important when evaluations are expensive, time-consuming, destructive, or performed in parallel batches.

A natural failure-aware strategy evaluates each candidate according to two marginal considerations: how likely it is to satisfy the target requirement and how likely it is to yield a valid experimental outcome. Their product
provides a marginal valid-hit score. A straightforward batch strategy is therefore to rank candidates independently by this score and select the highest-ranked ones. We refer to this top-$B$ marginal strategy as the \textsc{POF} baseline, retaining the method label used in our companion implementation.\footnote{Throughout this paper, \textsc{POF} is used only as the name of the implemented top-$B$ marginal valid-hit baseline. It should not be confused with the conventional probability of feasibility in constrained Bayesian optimization, which corresponds to the validity probability
introduced formally in Section~\ref{sec:methodology}.} If the marginal scores were calibrated valid-hit probabilities and candidate outcomes were conditionally independent, this strategy would maximize the current-round probability of obtaining at least one successful design.

In practice, the independence and calibration assumptions can be unreliable in small-data inverse design at the batch level. Several candidates may receive high marginal scores because they are supported by the same uncertain target-landscape explanation. Selecting all of them can produce a batch with strong individual scores but limited protection against model misspecification: if the shared explanation is incorrect, the entire batch may fail to discover the target. This distinction is particularly consequential for first-hit discovery, for which the value of a batch depends primarily on whether at least one member succeeds rather than on the sum of its marginal candidate utilities.

\paragraph{Related Work and Positioning.}
Batch Bayesian optimization has long recognized that parallel recommendations should not be obtained solely by repeating a sequential acquisition rule, but should account for interactions among batch members. Existing approaches model batch interactions through, for example, simulation matching, jointly maximizing information or value of information, locally penalizing neighborhoods of selected candidates, or explicitly encouraging diversity through determinantal point processes
\citep{azimi2010batch,shah2015parallel,gonzalez2016batch,kathuria2016batched,wu2016parallel}.
These methods establish the importance of batch-level complementarity, but they are generally developed for conventional black-box optimization objectives and do not specifically address valid-target first-hit discovery under unknown experimental failure.

Constrained Bayesian optimization models unknown feasibility and modifies acquisition decisions according to the probability or information value of satisfying the constraints
\citep{gardner2014bayesian,hernandez2015predictive,bergmann2020safe}.
In scientific experimentation, failed evaluations may instead produce no usable target measurement, introducing a second source of difficulty. Prior work has addressed this setting through pessimistic padding of missing responses, learned feasibility classifiers, explicit exploration of feasible--infeasible boundaries, and models of stochastic observation failure
\citep{wakabayashi2022bayesian,tian2024boundary,
hickman2025anubis,iwazaki2025no}.
These methods primarily improve candidate-level feasibility awareness. They do not directly determine whether the valid-hit opportunities represented by simultaneously selected candidates are complementary.

The objective of discovering target members is also related to active search, which allocates queries to identify as many members of a valuable class as possible \citep{garnett2012bayesian,jiang2017efficient,lookman2019active}. Related batch active-learning work uses adaptive submodularity to characterize diminishing returns in parallel information acquisition \citep{chen2013near}. ARC-SC differs in both objective and construction: it targets the first
valid target hit, preserves a strong marginal-score component, and uses submodular predictive-scenario coverage to diversify only the remaining batch positions under risk-support control.

In summary, the resulting gap lies at the intersection of first-hit search, batch complementarity, and failure awareness. A suitable acquisition rule should preserve candidates with a high marginal probability of immediate success, avoid allocating the full batch to candidates supported by overlapping predictive scenarios, and control exposure to candidates with low predicted validity or weak support from the observed design distribution.

We address this problem with ARC-SC, an Anchored Risk-Constrained Scenario Coverage method for failure-aware first-hit batch inverse design. ARC-SC constructs each batch in two stages. It first selects a set of anchors using
the same marginal valid-hit criterion as the \textsc{POF} baseline, subject to ARC-SC's risk-admissibility rule. It then fills the remaining batch positions by maximizing a noisy-OR scenario-coverage objective over multiple predictive target scenarios. Candidate responses are modulated by rank-normalized marginal valid-hit, validity, and empirical-support
information, while a batch-level risk-support constraint limits unsafe or strongly extrapolative recommendations. Scenario completion therefore complements rather than replaces marginal exploitation.

We show that the ARC-SC scenario-coverage objective is normalized, monotone, and submodular. Consequently, conditional on fixed anchors and a fixed admissible completion pool, greedy scenario completion obtains the classical approximation guarantee for cardinality-constrained monotone submodular maximization. We explicitly separate this fixed-domain result from the full implementation, in which the admissible set changes with the partial batch through a dynamic risk constraint.

We evaluate ARC-SC through frozen-oracle closed-loop simulations constructed from the UCI Superconductivity and JARVIS CFID materials-property datasets. Across four controlled failure mechanisms, three invalid-rate levels, three initialization conditions, and paired random seeds, ARC-SC produces more earlier first-hit outcomes than later outcomes relative to ordinary batch POF and increases cumulative valid target discoveries. Its clearest gains occur under feature-interaction and low-validity failures, whereas many hybrid and off-manifold configurations result in ties. The experiments therefore support ARC-SC as a POF-preserving diversification strategy whose benefit is concentrated in settings where structured failure and predictive redundancy matter.

Our contributions are:
\begin{enumerate}
\item We formulate failure-aware batch inverse design around the first valid target hit and propose ARC-SC, which combines marginal valid-hit anchoring, predictive scenario coverage, and batch-level risk-support control.
\item We prove that the scenario-coverage objective is normalized, monotone, and submodular, and derive the corresponding fixed-domain greedy-completion guarantee.
\item We evaluate ARC-SC on two frozen-oracle materials benchmarks under controlled failure mechanisms, showing more earlier first-hit wins than losses relative to ordinary batch POF, with the strongest gains under feature-interaction and low-validity failures.
\end{enumerate}

\section{Methodology}
\label{sec:methodology}

\subsection{Problem Formulation}
We consider a closed-loop batch inverse-design problem in which experimental evaluations are sequentially performed to discover a design satisfying predefined target requirements. At iteration $t$, the optimizer has access to an observed dataset $\mathcal D_t=\{(x_i,y_i,v_i)\}_{i=1}^{n_t}$, where $x_i\in\mathcal X$ denotes a candidate design, $v_i\in\{0,1\}$ indicates whether the experiment returns a valid measurement, $y_i$ denotes the measured target property when $v_i=1$ and is
unavailable when $v_i=0$. In a practical setting, the unavailability of target property could be caused by the candidate not being successfully synthesized, processed, or measured.

The objective is to identify, as early as possible, at least one candidate satisfying both validity and target requirements: $v(x)=1, y(x)\ge T $, for a user-defined target value $T$. Unlike conventional optimization objectives that maximize the final objective value or cumulative discoveries, our goal is first-hit discovery. Given a batch size $B$, the optimizer selects $\mathcal B_t=\{x_{t,1},\ldots,x_{t,B}\}$ at each experimental round. 

ARC-SC is designed to increase the probability of obtaining a valid target hit in the current batch and thereby promote earlier discovery. We use the first-hit round as the primary performance criterion:

$$
\tau_{\mathrm{hit}}=\min\{t:\exists x\in\mathcal B_t,v(x)=1,y(x)\ge T\},
$$

with the convention $\tau_{\mathrm{hit}}=\infty$ if no valid target hit is observed. Since experiments are conducted in batches, the utility of a candidate cannot be evaluated independently: multiple candidates with high individual acquisition scores may still provide limited probability of discovering a successful design if they represent the same uncertain region of the design space. ARC-SC addresses this issue by explicitly optimizing batch-level complementary coverage under predictive uncertainty. 

\subsection{Anchored Risk-Constrained Scenario Coverage}

ARC-SC decomposes batch construction into two complementary selection stages. The first stage preserves candidates with high marginal POF scores, while the second allocates the remaining batch positions to candidates that cover complementary predictive scenarios.

The resulting batch is constructed as $\mathcal B_t=\mathcal B_{\mathrm{anchor}}\cup\mathcal B_{\mathrm{cover}},$ where $\mathcal B_{\mathrm{anchor}}$ contains high-confidence candidates selected using POF, and $\mathcal B_{\mathrm{cover}}$ contains candidates selected according to posterior scenario coverage.

The overall workflow of ARC-SC is depicted in Fig~\ref{fig:algo-workflow} and the individual components are described below.

\begin{figure*}[t]
\centering

\resizebox{0.98\textwidth}{!}{%
\begin{tikzpicture}[
    font=\small,
    >=Latex,
    workflowbox/.style={
        draw=gray!70,
        rounded corners=4pt,
        line width=0.8pt,
        fill=gray!3,
        text width=3.25cm,
        minimum height=4.05cm,
        inner sep=6pt
    },
    visualbox/.style={
        workflowbox,
        align=center
    },
    textworkflowbox/.style={
        workflowbox,
        align=center
    },
    stepnum/.style={
        circle,
        fill=gray!65,
        text=white,
        minimum size=7.5mm,
        inner sep=0pt,
        font=\bfseries
    },
    flow/.style={
        -{Latex[length=2.8mm,width=1.8mm]},
        line width=1pt,
        draw=gray!70
    }
]

\node[textworkflowbox] (n1) at (-6.9,0) {
    \textbf{Learn from Previous Experiments}\\[7pt]
    {\color{gray!65}\Large
    \(\circ\quad\times\quad\circ\quad\times\)}\\[4pt]
    {\scriptsize Valid and failed experimental outcomes}\\[6pt]
    \rule{2.35cm}{0.35pt}\\[5pt]
    {\scriptsize
    Fit the target and validity models and update predictive
    uncertainty}
};

\node[stepnum]
at ([xshift=-1mm,yshift=1mm]n1.north west) {1};

\node[visualbox] (n2) at (-2.30,2.55) {
    \textbf{Generate and Score Candidates}\\[2pt]
    \includegraphics[
        width=2.68cm
    ]{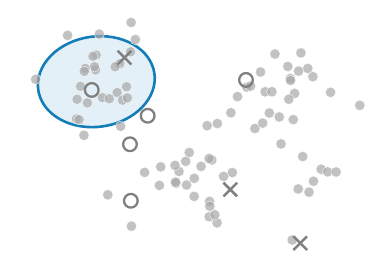}\\[-1pt]
    {\scriptsize
    Construct a finite pool and identify the current
    high-promise region}
};

\node[stepnum]
at ([xshift=-1mm,yshift=1mm]n2.north west) {2};

\node[visualbox] (n3) at (2.30,2.55) {
    \textbf{Keep Strong Marginal Anchors}\\[2pt]
    \includegraphics[
        width=2.68cm
    ]{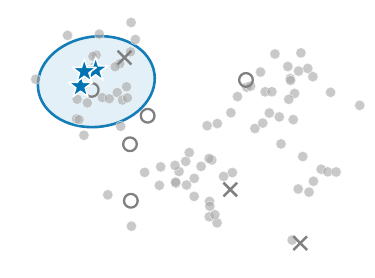}\\[-1pt]
    {\scriptsize
    Preserve candidates with strong marginal valid-hit
    potential}
};

\node[stepnum]
at ([xshift=-1mm,yshift=1mm]n3.north west) {3};

\node[visualbox] (n4) at (6.90,2.55) {
    \textbf{Construct Predictive Scenarios}\\[2pt]
    \includegraphics[
        width=2.68cm
    ]{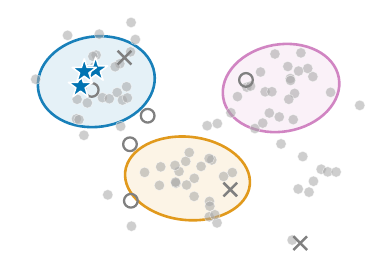}\\[-1pt]
    {\scriptsize
    Represent alternative plausible locations of the target
    opportunity}
};

\node[stepnum]
at ([xshift=-1mm,yshift=1mm]n4.north west) {4};

\node[visualbox] (n5) at (6.90,-2.55) {
    \textbf{Measure Scenario Coverage}\\[2pt]
    \includegraphics[
        width=2.68cm
    ]{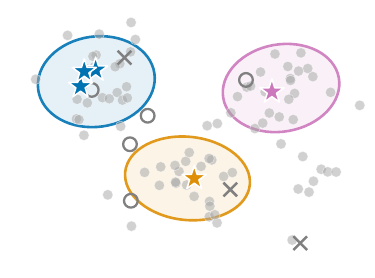}\\[-1pt]
    {\scriptsize
    Add candidates representing scenarios not already covered
    by the anchors}
};

\node[stepnum]
at ([xshift=-1mm,yshift=1mm]n5.north west) {5};

\node[visualbox] (n6) at (2.30,-2.55) {
    \textbf{Build a Risk-Aware Batch}\\[2pt]
    \includegraphics[
        width=2.68cm
    ]{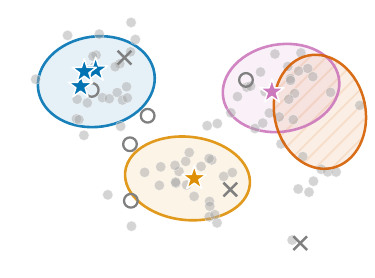}\\[-1pt]
    {\scriptsize
    Retain complementary candidates under the validity-support
    risk budget}
};

\node[stepnum]
at ([xshift=-1mm,yshift=1mm]n6.north west) {6};

\node[visualbox] (n7) at (-2.30,-2.55) {
    \textbf{Evaluate and Update}\\[2pt]
    \includegraphics[
        width=2.68cm
    ]{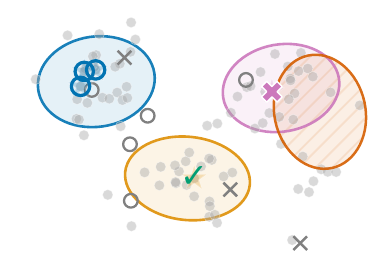}\\[-1pt]
    {\scriptsize
    Observe valid hits, valid non-hits, and failures, then
    append the outcomes}
};

\node[stepnum]
at ([xshift=-1mm,yshift=1mm]n7.north west) {7};

\draw[flow]
    (n1.north east)
    to[out=35,in=180]
    (n2.west);

\draw[flow]
    (n2.east) -- (n3.west);

\draw[flow]
    (n3.east) -- (n4.west);

\draw[flow]
    (n4.south) -- (n5.north);

\draw[flow]
    (n5.west) -- (n6.east);

\draw[flow]
    (n6.west) -- (n7.east);

\draw[flow]
    (n7.west)
    to[out=180,in=-35]
    (n1.south east);

\end{tikzpicture}%
}

\vspace{1.5mm}

{\scriptsize
\setlength{\tabcolsep}{7pt}
\renewcommand{\arraystretch}{1.25}
\begin{tabular}{lll}
\arccandidate\ Candidate
&
\arcvalidnonhit\ Previous valid non-hit
&
\arcinvalid\ Previous invalid
\\
\arcstar{arcblue}\ POF anchor selection
&
\arcstar{arcorange}\hspace{-0.7mm}
\arcstar{arcpurple}\ Scenario-coverage selection
&
\arcscenarios\ Predictive target scenarios
\\
\arcrisk\ Predicted high-risk region
&
\archit\ Valid target hit
&
\(\circ/\times\)\ Experimental non-hit / invalid outcome
\end{tabular}
}

\caption{
Integrated schematic overview of ARC-SC. At each closed-loop iteration, target and validity models are updated from previous valid and failed experiments and used to score a finite candidate pool. Strong marginal valid-hit candidates are retained through POF-style risk-admissible anchoring. Multiple predictive target scenarios then expose alternative plausible opportunities, and the remaining batch positions are assigned to complementary candidates under a validity-support risk budget. Experimental evaluation reveals valid target hits, valid non-hits, and invalid outcomes, which are
appended to the observed data for the next iteration. The embedded design-space illustrations are schematic and do not represent a particular experimental run.
}
\label{fig:algo-workflow}

\end{figure*}
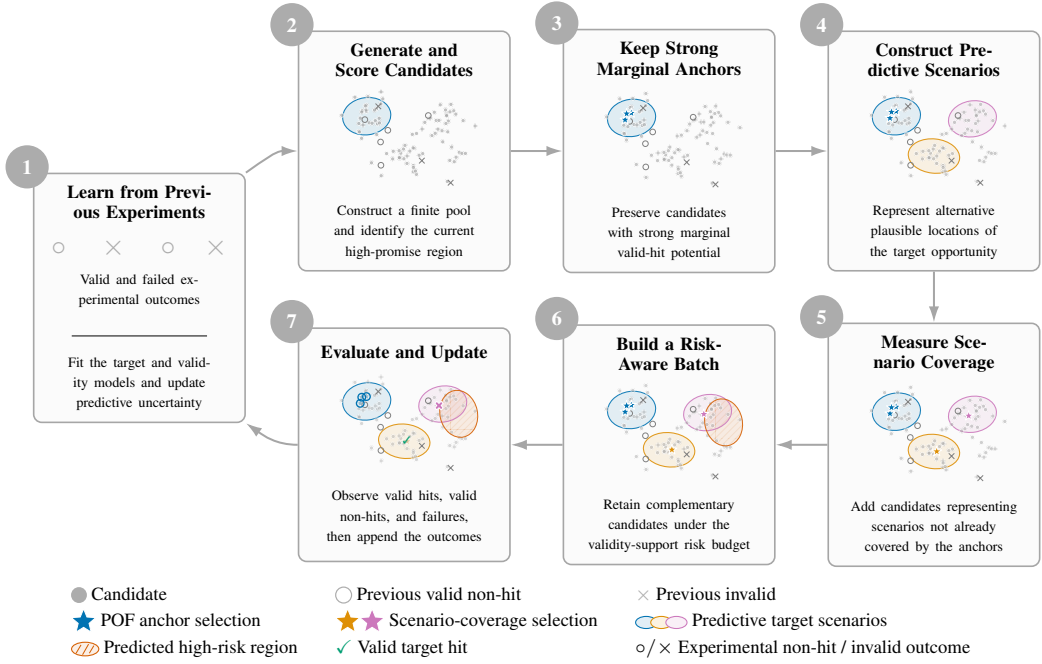

\subsection{Candidate Modeling and Risk Estimation}
At each iteration, ARC-SC first constructs a finite candidate pool $\mathcal C_t=\{x_j\}_{j=1}^{N}\subseteq\mathcal X$, over which the acquisition procedure is evaluated. All candidate-level
quantities and batch-selection decisions described below are computed on $\mathcal C_t$. For notational simplicity, we write $\mathcal C$ when the iteration index is clear from context. The candidate-generation procedure and its experimental configuration are described in Appendix.

\paragraph{Target model.}
The target surrogate is trained using only experimentally valid observations $\mathcal D_t^{y}=\{(x_i,y_i):v_i=1\}$. For each candidate $x$, the surrogate provides a predictive distribution over the unknown target property $p(y(x)|x,\mathcal D_t)$. The probability that a candidate satisfies the target threshold is

\begin{equation}
p_y(x)=P(y(x)\ge T|x,\mathcal D_t).
\end{equation}

This quantity is the surrogate-estimated posterior probability that $x$ satisfies the target threshold.

\paragraph{Validity model.}
Failed experiments provide binary information about the unknown feasible region. ARC-SC therefore trains a validity model using both successful and failed evaluations $\mathcal D_t^{v}=\{(x_i,v_i)\}_{i=1}^{n_t}$. The model estimates

\begin{equation}
p_v(x)=P(v(x)=1|x,\mathcal D_t).
\end{equation}

The quantity $p_v(x)$ is the predicted validity probability; in the terminology of constrained Bayesian optimization, it corresponds to the probability of feasibility.

\paragraph{Marginal valid-hit score and POF baseline.}
A successful discovery requires a candidate to both yield a valid measurement and satisfy the target requirement. We therefore define the product-form marginal valid-hit score
\begin{equation}
h(x)=p_y(x)p_v(x).
\label{eq:valid-hit-score}
\end{equation}
Under conditional independence between target satisfaction and validity, $h(x)$ equals the corresponding joint posterior probability; more generally, we use it as a marginal acquisition score.

Our primary comparison method, denoted \textsc{POF}, constructs a batch by selecting the $B$ candidates with the largest marginal valid-hit scores:
\begin{equation}
\mathcal B_{\mathrm{POF}}
=\operatorname{TopB}_{x\in\mathcal C} h(x).
\label{eq:pof-baseline}
\end{equation}
Thus, \textsc{POF} is an independent marginal-ranking baseline: each batch position is selected according to $h(x)$ without explicitly accounting for redundancy among simultaneously selected candidates.

\paragraph{Empirical support estimation.}
Validity prediction alone may be unreliable in small-data regimes, especially for candidates far outside the observed design distribution. ARC-SC therefore incorporates an empirical support score. Let $d_{\mathrm{kNN}}(x)$ denote the average distance between candidate $x$ and its nearest observed designs in normalized feature space. We define $p_s(x)=\exp(-\frac{d_{\mathrm{kNN}}(x)^2}{2\tau_s^2})$, where $\tau_s$ controls the decay scale. The support score is high for candidates located near previously observed designs and decreases for extrapolative candidates.

\paragraph{Risk-support score.}
A candidate has low risk when it is both predicted to be valid and sufficiently supported by the observed design distribution. We define the candidate-level risk-support score as

$$
c(x)=1-p_v(x)p_s(x).
$$

For a nonempty partial batch $\mathcal S$, ARC-SC evaluates the average batch risk:

$$
C(\mathcal S)=\frac{1}{|\mathcal S|}\sum_{x\in \mathcal S}c(x).
$$

During batch construction, ARC-SC enforces $C(\mathcal S)\le R$ whenever a risk-feasible extension exists, and $R$ is a user-defined risk budget. If no feasible extension is
available, a penalized fallback rule is used to complete the batch.

\subsection{POF-Anchored Candidate Selection}

A central design choice in ARC-SC is that scenario coverage does not replace marginal exploitation. ARC-SC uses the same marginal valid-hit score $h(x)$ that defines the \textsc{POF} baseline to select its first $K$ batch members, while additionally enforcing ARC-SC's risk-admissibility rule. We refer to these candidates as \emph{POF anchors}.

Specifically, define
$$
\mathcal F_R(\mathcal S)
=\left\{
x\in\mathcal C\setminus\mathcal S:
C(\mathcal S\cup\{x\})\le R
\right\}.
$$

Starting from $\mathcal S_0^{\mathrm a}=\varnothing$, ARC-SC select 
$$
a_k\in\arg\max_{x\in\mathcal F_R(\mathcal S_{k-1}^{\mathrm a})}h(x),\quad
\mathcal S_k^{\mathrm a}=\mathcal S_{k-1}^{\mathrm a}\cup\{a_k\},
$$

for $k=1,\ldots,K$, whenever a feasible extension exists. We define $\mathcal B_{\mathrm{anchor}}=\mathcal S_K^{\mathrm a}$. These anchors preserve candidates with the strongest marginal probability of immediate success.

Because anchor selection is subject to the dynamic risk constraint, the POF anchors are not necessarily identical to the first $K$ members of $\mathcal B_{\mathrm{POF}}$. Rather, they preserve the same marginal selection criterion within the currently risk-admissible set.

The motivation is that when the predictive model is reliable, marginal valid-hit ranking may already provide strong first-hit performance. Completely replacing this exploitation component with uncertainty-driven diversification may therefore be counterproductive. POF anchoring preserves strong marginal opportunities while reserving the remaining batch capacity for scenario-based diversification.

\subsection{Predictive Scenario Construction}
After selecting anchors, ARC-SC models uncertainty in the unknown target landscape through multiple predictive scenarios. Ideally, one would sample complete posterior functions $f^{(1)},...,f^{(M)}\sim P(f|\mathcal D_t)$. For computational scalability, the current implementation constructs
marginal posterior-predictive scenarios on the finite candidate pool:

\begin{equation}
f^{(m)}(x_j)=\mu_j+\sigma_{y,j}z_{m,j}, \quad z_{m,j}\overset{\mathrm{iid}}{\sim}\mathcal N(0,1),
\end{equation}

where $\mu_j=\mu(x_j)$ and $\sigma_{y,j}=\sigma_y(x_j)$ denote the predictive mean and standard deviation returned by the target
surrogate for each $x_j\in\mathcal C$. These scenarios preserve the marginal predictive uncertainty at each candidate but do not explicitly retain cross-candidate posterior covariance. Further implementation details are provided in Appendix.

For scenario $m$, a candidate's target response is measured relative to a relaxed threshold

$$
T_m^{\mathrm{relax}}=\min\left(T,Q_{1-\pi}\{f^{(m)}(x):x\in\mathcal C\}\right),
$$

where $Q_{1-\pi}$ is the empirical $(1-\pi)$-quantile of sampled target values in scenario m, and $\pi\in(0,1)$. The relaxation avoids degenerate scenarios where no candidate appears promising early in the optimization process. The soft target response is then defined as

$$
\psi_m(x)=\operatorname{sigm}\left(\frac{f^{(m)}(x)-T_m^{\mathrm{relax}}}{\tau}\right),
$$

where $\operatorname{sigm}(u)=\frac{1}{1+\exp(-u)}$ is the logistic function, and $\tau>0$.

\subsection{Scenario Coverage Acquisition}

The central idea of ARC-SC is that a batch should not only contain individually promising candidates, but should cover multiple plausible predictive scenarios. For each scenario $m$, we define a candidate response\footnote{Although validity and support appear in both the scenario response and
the risk score, they play distinct roles: the rank gates attenuate the scenario utility of weak candidates, whereas $C(S)$ controls the aggregate risk exposure of the selected batch.}:

\begin{equation}
r_m(x)=g_h(x)g_v(x)g_s(x)\psi_m(x).
\end{equation}

To reduce sensitivity to probability miscalibration, we transform the marginal valid-hit, validity, and support scores into percentile-rank gates. Let $R_h(x)$, $R_v(x)$, and $R_s(x)$ denote the percentile ranks of $h(x)$, $p_v(x)$, and $p_s(x)$ within $\mathcal C$, respectively. We define
\begin{equation}
g_q(x)=\left[\epsilon_q+(1-\epsilon_q)R_q(x)\right]^{\alpha_q}, \quad q\in\{h,v,s\}.
\end{equation}
The floor $\epsilon_q>0$ prevents complete elimination of low-ranked candidates, while $\alpha_q$ controls the strength of the corresponding gate. Exact parameter settings are provided in
Appendix.

For a candidate batch $\mathcal B$, ARC-SC defines the scenario coverage objective:

\begin{equation}
A_{\mathrm{SC}}(\mathcal B)
=\frac1M\sum_{m=1}^{M}\left[1-\prod_{x\in \mathcal B}
(1-r_m(x))\right].
\label{eq:theory-scenario-coverage}
\end{equation}

For each predictive scenario $m$, $1-\prod_{x\in \mathcal B}(1-r_m(x))$ measures whether the batch contains at least one candidate with a strong response in that scenario. Consequently, redundant candidates provide diminishing additional value. If a selected candidate already covers scenario $m$, another candidate with a similar response contributes little additional coverage. The marginal gain of adding candidate $x$ to an existing partial batch $\mathcal S$ is

\begin{equation}
\Delta_{\mathrm{SC}}(x|\mathcal S)=\frac1M\sum_{m=1}^{M}r_m(x)\prod_{s\in \mathcal S}(1-r_m(s)).
\end{equation}

The marginal-gain expression admits a direct interpretation:
$r_m(x)$ measures the strength of candidate $x$ under scenario $m$, whereas $\prod_{s\in \mathcal S}(1-r_m(s))$ is the scenario mass not yet covered by the current partial batch.

\subsection{Risk-Constrained Greedy Batch Optimization}
Direct optimization of $A_{\mathrm{SC}}(\mathcal B)$ is combinatorial. ARC-SC therefore constructs the batch greedily: After selecting the anchor set $\mathcal S=\mathcal B_{\mathrm{anchor}}$, the remaining candidates are selected sequentially. At each step\footnote{If no candidate satisfies the risk budget, ARC-SC invokes the risk-penalized fallback rule defined in Appendix.}:

$$
x^\star\in\arg\max_{x\in\mathcal F_R(\mathcal S)}\Delta_{\mathrm{SC}}(x\mid \mathcal S).
$$

 The selected candidate is added $\mathcal S\leftarrow \mathcal S\cup\{x^\star\}$. This process continues until $|\mathcal S|=B$. Specifically, $\mathcal S_0=\mathcal B_{\mathrm{anchor}}.$ For $\ell=1,\ldots,B-K$, we select $x_\ell^\star\in\arg\max_{x\in\mathcal F_R(\mathcal S_{\ell-1})}\Delta_{\mathrm{SC}}(x\mid\mathcal S_{\ell-1})$ and update $\mathcal S_\ell=\mathcal S_{\ell-1}\cup\{x_\ell^\star\}$. Then we define $\mathcal B_{\mathrm{cover}}
=\{x_1^\star,\ldots,x_{B-K}^\star\}$,
$\mathcal B_t=\mathcal S_{B-K}.$

The ARC-SC algorithm is summarized as Algorithm \ref{alg:arc-sc}.

\begin{algorithm}[t]
\caption{ARC-SC at closed-loop iteration $t$}
\label{alg:arc-sc}
\begin{algorithmic}[1]
\Require Observed data $\mathcal D_t$, candidate pool $\mathcal C$,
batch size $B$, number of anchors $K$, number of scenarios $M$,
risk budget $R$
\Ensure Recommended batch $\mathcal B_t$

\State Fit the target surrogate on $\mathcal D_t^{y}$ and the validity
model on $\mathcal D_t^{v}$
\State Compute $p_y(x)$, $p_v(x)$, $p_s(x)$, $h(x)$, and $c(x)$
for all $x\in\mathcal C$
\State Compute rank gates $g_h(x)$, $g_v(x)$, and $g_s(x)$

\State Initialize $\mathcal S\gets\varnothing$

\Comment{Stage 1: POF anchors}
\For{$k=1,\ldots,K$}
    \State $\mathcal F_R(\mathcal S)\gets
    \{x\in\mathcal C\setminus\mathcal S:
    C(\mathcal S\cup\{x\})\le R\}$
    \State Select $a_k$ by maximizing $h(x)$ over
    $\mathcal F_R(\mathcal S)$; use the risk-penalized fallback if
    $\mathcal F_R(\mathcal S)=\varnothing$
    \State $\mathcal S\gets\mathcal S\cup\{a_k\}$
\EndFor
\State $\mathcal B_{\mathrm{anchor}}\gets\mathcal S$

\Comment{Predictive scenario construction}
\State Generate $\{f^{(m)}\}_{m=1}^{M}$ on $\mathcal C$
\State Compute $T_m^{\mathrm{relax}}$, $\psi_m(x)$, and
$r_m(x)$ for all $m$ and $x\in\mathcal C$

\Comment{Stage 2: scenario-coverage completion}
\For{$\ell=1,\ldots,B-K$}
    \State Compute $\Delta_{\mathrm{SC}}(x\mid\mathcal S)$
    for all $x\in\mathcal C\setminus\mathcal S$
    \State $\mathcal F_R(\mathcal S)\gets
    \{x\in\mathcal C\setminus\mathcal S:
    C(\mathcal S\cup\{x\})\le R\}$
    \State Select $x_\ell^\star$ by maximizing
    $\Delta_{\mathrm{SC}}(x\mid\mathcal S)$ over
    $\mathcal F_R(\mathcal S)$; use the risk-penalized fallback if
    $\mathcal F_R(\mathcal S)=\varnothing$
    \State $\mathcal S\gets\mathcal S\cup\{x_\ell^\star\}$
\EndFor

\State $\mathcal B_{\mathrm{cover}}\gets
\mathcal S\setminus\mathcal B_{\mathrm{anchor}}$
\State \Return $\mathcal B_t\gets
\mathcal B_{\mathrm{anchor}}\cup\mathcal B_{\mathrm{cover}}$
\end{algorithmic}
\end{algorithm}

\section{Theoretical Analysis}
\label{sec:theory}
In this section our analysis concerns the one-round batch-selection problem. We analyze ARC-SC at a fixed closed-loop iteration, conditioning on the observed dataset and the finite candidate pool. The predictive scenarios, percentile-rank gates, and candidate responses are therefore treated as fixed during construction of the current batch. In particular, for every candidate $x\in\mathcal C$ and predictive scenario $m\in\{1,\ldots,M\}$, the scenario response defined in Section~\ref{sec:methodology} satisfies $0\le r_m(x)\le 1$.

\subsection{From Marginal Valid-Hit Selection to Batch Scenario Coverage}

For a candidate $x$, let $H_x=\{v(x)=1\}\cap\{y(x)\ge T\}$ denote the event that $x$ is a valid target hit. The ideal current-round utility of a batch $\mathcal B$ is the posterior probability that at least one member of the batch succeeds:
\begin{equation}
F(\mathcal B)=\Pr\left(\bigcup_{x\in\mathcal B}H_x
\,\middle|\,\mathcal D_t\right).
\end{equation}

Suppose that the marginal valid-hit score $h(x)=p_y(x)p_v(x)$ is a calibrated marginal probability of $H_x$ and that the hit events are conditionally independent across candidates. Under these assumptions,
\begin{equation}
F(\mathcal{B})=1-\prod_{x\in\mathcal{B}}\bigl(1-h(x)\bigr).
\label{eq:independent-hit}
\end{equation}
For a fixed batch size $B$, \eqref{eq:independent-hit} is maximized by selecting the $B$ largest values of $h(x)$, exactly recovering the \textsc{POF} baseline defined in~\eqref{eq:pof-baseline}.

In small-data inverse design, however, several candidates with high marginal valid-hit scores may be supported by similar predictive assumptions. Their marginal scores can therefore be large while their batch-level contributions
are redundant. ARC-SC addresses this regime through the scenario-coverage surrogate \eqref{eq:theory-scenario-coverage}.
Because $r_m(x)$ includes a relaxed target response and rank-normalized gates, $A_{\mathrm{SC}}$ is not asserted to be a calibrated estimate of $F$. Rather, it is a bounded batch utility designed to reward coverage of multiple predictive scenarios while discounting repeated coverage of the same scenarios.

\subsection{Submodularity of Scenario Coverage}

The following proposition gives the central structural property of ARC-SC.

\begin{proposition}
\label{prop:submodularity}
The set function
$$
A_{\mathrm{SC}}:2^{\mathcal C}\rightarrow[0,1]
$$
defined in \eqref{eq:theory-scenario-coverage} is normalized, monotone, and submodular.
\end{proposition}

\begin{proof}
Normalization follows from the empty-product convention $A_{\mathrm{SC}}(\varnothing)=0$. For a partial batch $\mathcal S\subseteq\mathcal C$ and a candidate
$x\in\mathcal C\setminus\mathcal S$, the marginal gain is

\begin{equation}
\begin{aligned}
\Delta_{\mathrm{SC}}(x \mid \mathcal{S})
=A_{\mathrm{SC}}\bigl(\mathcal{S}\cup\{x\}\bigr)-A_{\mathrm{SC}}(\mathcal{S})
=\frac{1}{M}\sum_{m=1}^{M}r_m(x)\prod_{s\in\mathcal{S}}
\bigl(1-r_m(s)\bigr).
\end{aligned}
\label{eq:theory-marginal}
\end{equation}

Since all factors in \eqref{eq:theory-marginal} are nonnegative, $\Delta_{\mathrm{SC}}(x\mid\mathcal S)\ge0$,
which establishes monotonicity.

Now consider $\mathcal S\subseteq\mathcal T\subseteq\mathcal C$ and $x\in\mathcal C\setminus\mathcal T$. Because $1-r_m(s)\in[0,1]$,
$$
\prod_{s\in\mathcal S}\bigl(1-r_m(s)\bigr)\ge
\prod_{s\in\mathcal T}\bigl(1-r_m(s)\bigr).
$$
Multiplying by $r_m(x)\ge0$ and averaging over $m$ gives
$$
\Delta_{\mathrm{SC}}(x\mid\mathcal S)
\ge
\Delta_{\mathrm{SC}}(x\mid\mathcal T),
$$
which is the diminishing-returns condition for submodularity.
\end{proof}

Proposition~\ref{prop:submodularity} holds for any fixed response matrix whose entries satisfy $r_m(x)\in[0,1]$. It therefore does not require the predictive scenarios to be exact samples from a correlated Gaussian-process posterior. The scenario-generation procedure affects the statistical interpretation of the objective, but not its submodular structure.

Equation~\ref{eq:theory-marginal} also explains how ARC-SC reduces batch redundancy. Define the uncovered mass of scenario $m$ after selecting $\mathcal S$ as
$U_m(\mathcal S)=\prod_{s\in\mathcal S}\bigl(1-r_m(s)\bigr)$.
Then
\begin{equation}
\Delta_{\mathrm{SC}}(x\mid\mathcal S)
=\frac{1}{M}\sum_{m=1}^{M}U_m(\mathcal S)r_m(x).
\end{equation}
A candidate receives high marginal value when it responds strongly in scenarios that remain weakly represented by the current batch. Once a scenario has been strongly covered, its uncovered mass becomes small, and additional candidates with a similar response profile provide little further gain.

For two candidates $x$ and $z$, this redundancy discount can be written explicitly as
$$
A_{\mathrm{SC}}(\{x,z\})=\bar r(x)+\bar r(z)-O(x,z),
$$
where $\bar r(x)=\frac{1}{M}\sum_{m=1}^{M}r_m(x)$ and $O(x,z)=\frac{1}{M}\sum_{m=1}^{M}r_m(x)r_m(z)$.
The overlap term $O(x,z)$ is large when the two candidates respond strongly in the same predictive scenarios. Thus, ARC-SC induces complementarity in predictive-scenario response space without requiring an explicit pairwise geometric-repulsion term between selected candidates. 

\subsection{Guarantee for Anchored Greedy Completion}

Let $\mathcal A=\mathcal B_{\mathrm{anchor}}$
be a fixed set of $K$ POF anchors, and let
$k=B-K$
be the number of remaining batch positions. Conditional on $\mathcal A$, define the residual coverage utility
\begin{equation}
 G_{\mathcal A}(\mathcal S)=A_{\mathrm{SC}}(\mathcal A\cup\mathcal S)-A_{\mathrm{SC}}(\mathcal A).
\label{eq:residual-coverage}   
\end{equation}
The subtraction ensures that $G_{\mathcal A}(\varnothing)=0$.

\begin{proposition}
\label{prop:residual-submodularity}
For any fixed anchor set $\mathcal A$, the residual objective
$G_{\mathcal A}$ is normalized, monotone, and submodular with respect to $\mathcal S$.
\end{proposition}

\begin{proof}
Adding the same fixed set $\mathcal A$ to every argument preserves both monotonicity and diminishing marginal returns of $A_{\mathrm{SC}}$. Subtracting the constant
$A_{\mathrm{SC}}(\mathcal A)$ establishes normalization without changing any marginal gain.
\end{proof}

Consider a fixed admissible completion pool $\mathcal F\subseteq\mathcal C\setminus\mathcal A$, and assume that
$1\le k=B-K\le|\mathcal F|.$ Let $\mathcal S_{\mathrm{greedy}}$ denote the set of $k$ candidates
obtained by repeatedly selecting an exact maximum-marginal-gain candidate from $\mathcal F$. Let
$\mathcal S^\star\in\arg\max_{\substack{
\mathcal S\subseteq\mathcal F\\|\mathcal S|\le k
}}G_{\mathcal A}(\mathcal S)$ be an optimal completion of the same anchor set over the same fixed candidate pool.

\begin{corollary}[Greedy anchored completion]
\label{cor:greedy-completion}
For fixed anchors and a fixed admissible completion pool,
\begin{equation}
\begin{aligned}
G_{\mathcal A}(\mathcal S_{\mathrm{greedy}})
\ge\left[
1-\left(1-\frac{1}{k}\right)^k
\right]
G_{\mathcal A}(\mathcal S^\star)
\ge
\left(1-\frac{1}{e}\right)
G_{\mathcal A}(\mathcal S^\star).
\end{aligned}
\label{eq:greedy-guarantee}
\end{equation}
\end{corollary}

\begin{proof}
By Proposition~\ref{prop:residual-submodularity},
$G_{\mathcal A}$ is a normalized, monotone, submodular set function. The result follows from the classical greedy approximation guarantee for monotone submodular maximization under a cardinality constraint \citep{nemhauser1978analysis}. A complete derivation is provided in Appendix.
\end{proof}

Corollary~\ref{cor:greedy-completion} gives a conditional
interpretation of the two-stage ARC-SC construction. Once the POF anchors are fixed, greedy scenario completion approximates the optimal residual scenario coverage achievable by the remaining $B-K$ positions
over the same fixed completion pool. The anchors preserve strong marginal opportunities, whereas the completion stage reduces scenario-response redundancy conditional on those anchors.

The guarantee is conditional on a fixed admissible completion pool. In the implemented ARC-SC procedure, the admissible set changes with the current partial batch through the average risk-support constraint, and a penalized fallback is used when no risk-feasible extension exists.
Consequently, Corollary~\ref{cor:greedy-completion} does not constitute a global approximation guarantee for the full dynamically risk-filtered algorithm. Further analysis is provided in Appendix.

\section{Experiments and Results}
\label{sec:experiments}

\subsection{Experimental Setup}
\label{sec:experimental-setup}

\paragraph{Materials benchmarks and frozen oracles.}
We evaluate ARC-SC on two materials-property benchmarks: the UCI Superconductivity dataset \citep{superconductivty_data_464} and the JARVIS CFID dataset \citep{choudhary2020joint}. The superconductivity task searches for designs with high critical temperature, whereas the JARVIS task searches for low formation energy per atom. For JARVIS, we negate the formation energy so that larger response values are preferred in both benchmarks.

Because closed-loop optimization can recommend candidates that are not rows of the original datasets, we construct a deterministic frozen response oracle for each benchmark. Original dataset labels are used only for oracle construction and validation; all subsequent closed-loop evaluations, including newly generated candidates, use the frozen oracle response. Table~\ref{tab:benchmark-summary} summarizes the benchmark environments. Data preprocessing, descriptor selection, oracle training, and held-out validation are detailed in Appendix~\ref{app:oracle-construction}.

\begin{table}[t]
\caption{Summary of the frozen-oracle benchmark environments. The JARVIS response is negative formation energy per atom, such that larger values are preferred.}
\label{tab:benchmark-summary}
\centering
\resizebox{\linewidth}{!}{%
\begin{tabular}{lrrrrr}
\toprule
\multicolumn{1}{c}{\bfseries DATASET}
& \multicolumn{1}{c}{\bfseries SAMPLES}
& \multicolumn{1}{c}{\bfseries FEATURES}
& \multicolumn{1}{c}{\bfseries THRESHOLD}
& \multicolumn{1}{c}{\bfseries ORACLE $R^2$}
& \multicolumn{1}{c}{\bfseries SPEARMAN $\rho$}
\\
\hline
Superconductivity
& 21,263
& 5
& 80.0
& 0.913
& 0.956
\\
JARVIS CFID
& 55,723
& 15
& 1.539
& 0.838
& 0.912
\\
\bottomrule
\end{tabular}%
}
\end{table}

The superconductivity oracle is an Extra-Trees regressor trained on five selected composition descriptors and achieves $R^2=0.913$ and Spearman $\rho=0.956$ on held-out data. The JARVIS oracle is a random-forest regressor trained on 15 selected CFID descriptors and achieves $R^2=0.838$ and $\rho=0.912$. The superconductivity target corresponds to the upper $15\%$ of the frozen-oracle response distribution. For JARVIS, the upper-$25\%$ threshold is used because more extreme thresholds produce an almost entirely no-hit benchmark.

\paragraph{Failure mechanisms.}
We augment each frozen oracle with four controlled invalid-output mechanisms: feature interaction, smoothly varying low validity, off-manifold failure, and a hybrid mechanism combining multiple sources of invalidity. Each mechanism is calibrated to nominal invalid rates of $0.2$, $0.3$, and $0.4$. Exact definitions and calibration procedures are given in Appendix~\ref{app:failure-mechanisms}.

\paragraph{Closed-loop protocol.} Each run begins with $n_0=50$ observations drawn from the bottom $65\%$, $75\%$, or $85\%$ of the corresponding frozen-oracle response distribution. We then perform five closed-loop rounds with batch size $B=5$, yielding a total budget of 25 newly evaluated candidates. ARC-SC and the \textsc{POF} baseline use identical target-surrogate and validity-model classes under the same experimental configuration. \textsc{POF} selects all $B=5$ candidates by descending marginal valid-hit score $h(x)=p_y(x)p_v(x)$. ARC-SC uses the same score to select $K=3$
risk-admissible POF anchors and allocates the remaining two batch positions through scenario-coverage optimization. To assess stochastic variability, we use ten held-out seed IDs ($5$--$14$). For each benchmark, we evaluate four failure mechanisms, three invalid-rate levels, and three initialization settings, yielding $4 \times 3 \times 3 \times 10 = 360$  paired evaluation cases. Uniform sampling over descriptor bounds is used for all superconductivity experiments and for the JARVIS feature-interaction and low-validity mechanisms\footnote{In the 15-dimensional JARVIS descriptor space, combining global box sampling with the off-manifold and hybrid failure mechanisms creates a nearly degenerate search regime, we therefore use a closed-loop empirical trust-region generator. See detailed explain in Appendix}. For the standard uniform-pool protocol, a single candidate pool of $N=10{,}000$ designs is generated for each configuration and round and supplied to both ARC-SC and POF, ensuring that differences arise solely from the acquisition strategy.

\paragraph{Evaluation and statistical analysis.}
Our primary metric is the first closed-loop round containing a newly queried valid target hit. Runs without a hit during the five-round horizon are assigned a censored value of six,
\begin{equation}
\widetilde{\tau}_{\mathrm{hit}}=
\begin{cases}
\tau_{\mathrm{hit}}, & \text{if a hit is found},\\
6, & \text{otherwise}.
\end{cases}
\label{eq:censored-first-hit}
\end{equation}
We report the paired difference
\begin{equation}
\Delta\widetilde{\tau}_{\mathrm{hit}}=\widetilde{\tau}_{\mathrm{hit}}^{\mathrm{ARC\text{-}SC}}-\widetilde{\tau}_{\mathrm{hit}}^{\mathrm{POF}},
\end{equation}
for which negative values favor ARC-SC. A paired case is counted as a win, tie, or loss according to whether this difference is negative, zero, or positive.

To quantify uncertainty, we report two complementary statistics. First, an exact two-sided sign test is applied to discordant first-hit pairs, testing whether ARC-SC wins and losses occur with equal probability. Second, we construct $95\%$ confidence intervals for paired effects using $10{,}000$ hierarchical bootstrap replicates that resample experimental-condition cells and then seed-level pairs within each sampled cell. As secondary endpoints, we report cumulative valid target hits and the invalid-evaluation rate. 

\subsection{Overall First-Hit Performance}
\label{sec:overall-results}

Table~\ref{tab:confirmatory-results} summarizes the aggregate paired results. On superconductivity, ARC-SC obtains 70 earlier first-hit outcomes, 251 ties, and 39 later outcomes. Among the 109 discordant pairs, ARC-SC wins $64.2\%$ of the comparisons, with an exact $95\%$ confidence interval of $[54.5\%,73.2\%]$ and a two-sided sign-test value of $p=0.004$. The mean paired censored first-hit effect is $\Delta\widetilde{\tau}_{\mathrm{hit}}=-0.269$, with hierarchical-bootstrap $95\%$ confidence interval $[-0.539,-0.025].$ Thus, both the win--loss imbalance and the average first-hit effect favor ARC-SC.

\begin{table}[t]
\centering
\caption{First-hit comparison against the \textsc{POF} baseline using ten seeds per condition. W/T/L denotes earlier-hit wins, ties, and later-hit losses. The conditional win fraction is computed over discordant pairs only. $\Delta\widetilde{\tau}_{\mathrm{hit}}<0$ favor ARC-SC. Intervals are $95\%$ confidence intervals.}
\label{tab:confirmatory-results}
\resizebox{\linewidth}{!}{
\begin{tabular}{lcccc}
\toprule
\bfseries Dataset
& \bfseries W/T/L
& \bfseries Cond. win frac.
& \bfseries Sign $p$
& $\Delta\widetilde{\tau}_{\mathrm{hit}}$\\
\midrule
Superconductivity
& 70/251/39
& $0.642\,[0.545,0.732]$
& $0.004$
& $-0.269\,[-0.539,-0.025]$\\
JARVIS CFID
& 49/272/39
& $0.557\,[0.447,0.663]$
& $0.337$
& $-0.117\,[-0.297,0.053]$\\
\bottomrule
\end{tabular}
}
\end{table}

The superconductivity improvement arises through both forms of first-hit advantage. ARC-SC finds a target when POF fails to find one within the budget in 41 cases, compared with 23 cases in the opposite direction; when both methods eventually succeed but in different rounds, ARC-SC is earlier in 29 cases and later in 16. ARC-SC also reduces the invalid-evaluation rate by an average of 5.9 percentage points, with a bootstrap interval entirely below zero. The cumulative number of valid target hits increases by 102 over the full grid, corresponding to a mean paired increase of $0.283$ hits per run; its $95\%$ interval $[-0.064,0.667]$ includes zero, indicating that the clearest supported benefit is first-hit timing rather than cumulative discovery yield.

On the mixed-pool JARVIS benchmark, ARC-SC obtains 49 earlier first-hit outcomes, 272 ties, and 39 later outcomes. The corresponding mean first-hit effect, $\Delta\widetilde{\tau}_{\mathrm{hit}}=-0.117$, $95\%\ \mathrm{CI}=[-0.297,0.053],$
is directionally favorable but statistically inconclusive. Likewise, the exact sign test does not reject equal win and loss probabilities. ARC-SC nevertheless remains competitive across the 360 cases, producing 18 additional valid target hits in aggregate and a small mean reduction in invalid evaluations; both secondary-effect intervals include zero. We therefore interpret JARVIS as evidence of robustness rather than statistically established superiority.

\subsection{Discovery Probability Across Experimental Rounds}
\label{sec:discovery-curves}

Figure~\ref{fig:confirmatory-first-hit-curves} shows the probability that a run has obtained at least one newly queried valid target hit by each experimental round. This visualization directly reflects the first-hit objective rather than cumulative optimization performance.

\begin{figure}[h]
\centering
\includegraphics[width=\linewidth]{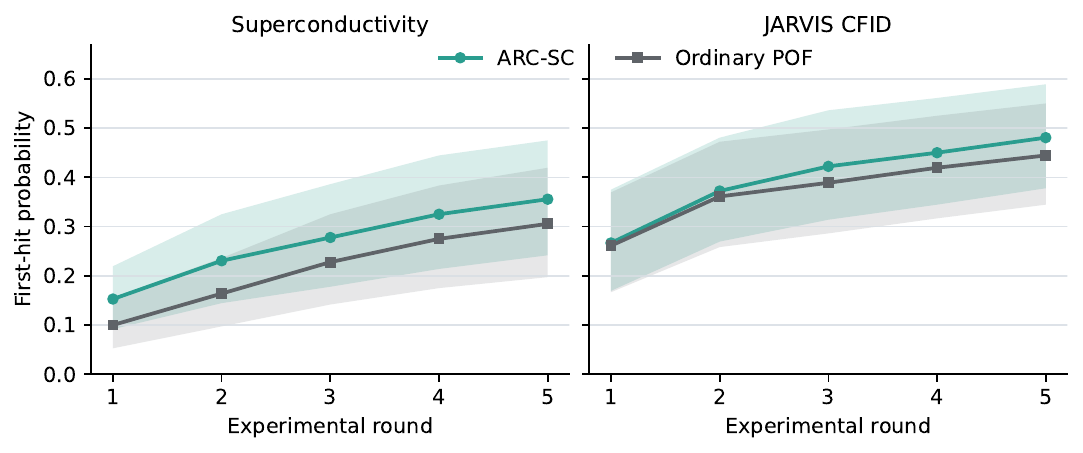}
\caption{First-hit probability as a function of experimental round. Lines show ARC-SC and \textsc{POF} baseline; shaded regions denote hierarchical-bootstrap $95\%$ confidence intervals. Left: superconductivity with shared uniform candidate pools. Right: JARVIS with the mixed candidate-pool protocol. The results show that ARC-SC improves early target discovery on superconductivity with statistically significant gains in first-hit performance. On JARVIS, ARC-SC remains competitive. Although it does not achieve statistically significant superiority, it demonstrates robustness in the more challenging high-dimensional design space.}
\label{fig:confirmatory-first-hit-curves}
\end{figure}

For superconductivity, ARC-SC separates from POF immediately: the round-1 discovery probabilities are $15.3\%$ and $10.0\%$, respectively, and by round 2 they are $23.1\%$ and $16.4\%$. The corresponding paired differences are $+5.3$ and $+6.7$ percentage points, with bootstrap intervals excluding zero. The ARC-SC curve remains numerically above POF throughout the five-round horizon, reaching $35.6\%$ compared with $30.6\%$ at round 5.

The JARVIS curves are closer. ARC-SC reaches a first-hit probability of $48.1\%$ by round 5, compared with $44.4\%$ for POF. The differences become progressively positive after the first round but their confidence intervals include zero. Importantly, the mixed candidate-pool protocol substantially reduces the degenerate no-hit behavior observed with global box sampling, allowing the off-manifold and hybrid mechanisms to contribute meaningful closed-loop comparisons rather than deterministic ties.

\subsection{Behavior Across Failure Mechanisms}
\label{sec:failure-results}

The experiments also reveal that the value of scenario coverage depends on the failure structure. On superconductivity, the clearest advantage occurs under feature-interaction failures: ARC-SC records 37 wins, 33 ties, and 20 losses across 90 paired cases, while producing 98 additional valid target hits and reducing the invalid rate by 13.5 percentage points on average. Under low-validity failures, ARC-SC retains a favorable first-hit balance of 26/45/19 and again reduces invalid evaluations, although cumulative-hit performance is essentially unchanged. Hybrid and off-manifold conditions are dominated by ties, indicating that the benefit of scenario completion is concentrated in regimes where multiple plausible target opportunities remain accessible.

For JARVIS, the mixed candidate-pool experiment produces nondegenerate comparisons across all four mechanisms. Under the trust-region protocol, ARC-SC obtains 10/73/7 first-hit W/T/L under off-manifold failure and 7/75/8 under the hybrid mechanism, while reducing invalid evaluations by approximately 10.1 and 5.2 percentage points, respectively. Under the shared uniform pool, feature-interaction and low-validity mechanisms yield 15/62/13 and 17/62/11 first-hit W/T/L. These mechanism-level results are descriptive rather than individually powered hypothesis tests.

\section{Conclusion}
\label{sec:conclusion}

We introduced ARC-SC, an anchored risk-constrained scenario-coverage method for first-hit batch inverse design under experimental failure. Rather than constructing a batch solely from marginal candidate scores, ARC-SC uses the
same marginal valid-hit criterion as the POF baseline to retain strong risk-admissible anchors, then allocates the remaining batch capacity to complementary predictive target scenarios. This construction preserves marginal exploitation while reducing redundancy when predictive uncertainty admits multiple plausible opportunities for target discovery.

We showed that the ARC-SC scenario-coverage objective is normalized, monotone, and submodular. Consequently, for fixed anchors and a fixed admissible completion pool, greedy scenario completion inherits the classical approximation guarantee for cardinality-constrained submodular maximization. The result formalizes the diminishing-return structure underlying scenario coverage: candidates are valuable when they respond strongly in predictive scenarios that remain insufficiently represented by the current batch. 

Holdout-seed experiments on frozen-oracle superconductivity and JARVIS materials benchmarks provide complementary evidence for the proposed approach. On superconductivity, ARC-SC achieves a statistically supported reduction in censored first-hit round relative to the POF baseline, and simultaneously reduces invalid evaluations, reproducing the favorable behavior observed during development on multiple independent runs. On the higher-dimensional JARVIS benchmark, ARC-SC remains competitive and directionally favorable in first-hit performance, but less statistically significant. Overall, these results support ARC-SC as a POF-preserving diversification strategy whose clearest benefit arises when meaningful alternative target opportunities remain accessible under predictive uncertainty, rather than as a uniformly superior replacement for marginal acquisition.

Several directions follow naturally from these results. First, batch diversification may be better treated as an adaptive decision rather than a fixed design choice. More faithful posterior-scenario modeling, together with adaptive control of the anchor count and risk budget, could allow the degree of diversification to respond to model confidence and estimated redundancy among high-scoring candidates. This direction also connects to a broader view of scientific modeling in which uncertainty should be propagated coherently across interacting components rather than treated locally within individual predictions \citep{wang2024bayesian}. For closed-loop discovery, this suggests jointly reasoning over target uncertainty, experimental validity, empirical support, and batch composition when deciding how much exploration is warranted. Second, the JARVIS experiments indicate that candidate generation and acquisition should not be viewed as independent components: when a high-dimensional proposal space is dominated by unsupported designs, candidate reachability can become more limiting than the acquisition rule itself. This observation connects scenario-aware selection with adaptive local and subspace-search methods \citep{eriksson2019scalable,papenmeier2022increasing}, and motivates policies that jointly adapt the reachable search region, risk allocation, and batch composition. Finally, the first-hit formulation points toward a broader objective for autonomous experimentation: optimizing the scientific discovery process under finite budgets, experimental failures, and parallel capacity, rather than only the terminal objective value \citep{garnett2012bayesian,volk2024performance,tom2024self}. Broader baseline comparisons and real closed-loop studies will be important for determining when such adaptive discovery policies provide the greatest practical benefit.

\bibliography{references}
\bibliographystyle{references}

\appendix

\section{Appendix}

\subsection{Implementation and Reproducibility}
\label{app:implementation}

This section gives the implementation details needed to reproduce the acquisition rules used in Section~\ref{sec:experiments}. We distinguish the shared predictive infrastructure from the batch-selection rules. ARC-SC and the \textsc{POF} baseline use the same target surrogate, validity model, observed data, target threshold, and---for the uniform-pool experiments---the same finite candidate pool. The \textsc{POF} baseline is the method defined in~\eqref{eq:pof-baseline}: it selects the $B$ largest marginal valid-hit scores $h(x)=p_y(x)p_v(x)$ and does not apply ARC-SC's support-based risk filter. ARC-SC instead uses $h(x)$ for risk-admissible anchoring and uses scenario coverage for the remaining positions.

\subsubsection{Uniform Candidate-Pool Construction}
\label{app:candidate-generation}

For the uniform-pool experiments, bounds are inferred independently for every selected descriptor from the complete benchmark reference feature matrix. If $\ell_j$ and $u_j$ are the observed minimum and maximum of descriptor $j$, the padded search interval is
\begin{equation}
\left[
\ell_j-\rho_b(u_j-\ell_j),\;
u_j+\rho_b(u_j-\ell_j)
\right],
\qquad \rho_b=0.05.
\label{eq:app-main-bounds}
\end{equation}
At every round, $N=10{,}000$ candidates are sampled independently and uniformly from this axis-aligned box. In a paired uniform-pool case, the pool is sampled once using the condition- and round-specific seed and is supplied unchanged to ARC-SC and \textsc{POF}. Thus, the two methods are compared on exactly the same candidates. This uniform protocol is used for all superconductivity experiments and for the JARVIS feature-interaction and low-validity mechanisms.

\subsubsection{Closed-Loop JARVIS Trust-Region Candidate Pool}
\label{app:jarvis-trust-region}

In the 15-dimensional JARVIS descriptor space, global uniform sampling becomes strongly dominated by unsupported candidates under the off-manifold and hybrid failure mechanisms. For these two mechanisms in the main mixed-pool JARVIS experiment, we therefore generate candidates from a closed-loop empirical trust region.

At the start of each round, let $\mathcal O_t$ denote the rows observed by the current method and let $\mathcal O_t^{v}\subseteq\mathcal O_t$ denote its valid rows. Candidate anchors are sampled uniformly with replacement from $\mathcal O_t^{v}$ whenever at least one valid row is available; otherwise all observed rows are used. For descriptor $j$, let $\operatorname{IQR}_{t,j}$ be the interquartile range among the currently observed designs. The local perturbation scale is
\begin{equation}
s_{t,j}
=
0.05\,\widetilde{\operatorname{IQR}}_{t,j},
\end{equation}
where a degenerate observed IQR is replaced first by the corresponding IQR of the benchmark reference features and, if necessary, by $10^{-3}$ times the full reference range. Given a sampled anchor $a$, a candidate is generated as
\begin{equation}
x_j=a_j+\epsilon_j,
\qquad
\epsilon_j\sim\mathcal N(0,s_{t,j}^2),
\label{eq:app-trust-region}
\end{equation}
then clipped to the unpadded minimum and maximum of the benchmark reference features. Duplicate candidate rows are removed.

The trust-region pool uses only feature values and observed validity outcomes available to that method at the start of the round. It does not use held-out target labels, oracle values for unqueried candidates, future observations, or observations generated by the competing method. Consequently, the trust-region pools are trajectory dependent and need not be identical between ARC-SC and \textsc{POF}. These cases therefore compare complete closed-loop policies rather than two acquisition rules on a fixed common pool. This distinction is explicit in our interpretation of the JARVIS results.

\subsubsection{Target and Validity Models}
\label{app:predictive-models}

At every round, the target surrogate is refitted using only valid observations. The reported experiments use a Gaussian-process regressor after standardizing the input descriptors. Its kernel is
\begin{equation}
k(x,x')
=
C\,k_{\mathrm{Mat\acute{e}rn},\nu=5/2}(x,x')
+
k_{\mathrm{white}}(x,x'),
\end{equation}
where the constant-kernel amplitude has bounds $[10^{-3},10^3]$, the initial Mat\'ern length scale is one, and the white-noise level is initialized at $10^{-3}$. The regressor normalizes the target internally and uses two optimizer restarts. Predictive means and standard deviations are denoted by $\mu(x)$ and $\sigma_y(x)$.

The validity model is fitted to all observed valid/invalid labels. The experiments use a random-forest classifier with 300 trees, square-root feature subsampling, balanced-subsample class weighting, and the condition-specific random seed. If only one validity class has been observed, the implementation uses a constant-probability model until both classes are available. ARC-SC and \textsc{POF} use the same model classes and model seeds within each paired case.

\subsubsection{Empirical-Support Score}
\label{app:support-score}

Let $\widetilde x$ denote a candidate after standardizing each descriptor using the mean and standard deviation of the currently observed designs. Let $\mathcal N_k(x)$ be its $k$ nearest observed designs in Euclidean distance. With $k=5$,
\begin{equation}
d_{\mathrm{kNN}}(x)
=
\frac{1}{k}
\sum_{x_i\in\mathcal N_k(x)}
\|\widetilde x-\widetilde x_i\|_2.
\label{eq:app-knn-distance}
\end{equation}
The scale $\tau_s$ is recomputed at each round as the empirical $75$th percentile of the corresponding five-neighbor distances evaluated on the observed designs, with a numerical lower bound of $10^{-6}$. The support score is
\begin{equation}
p_s(x)
=
\exp\!\left[
-\frac{d_{\mathrm{kNN}}(x)^2}{2\tau_s^2}
\right].
\end{equation}
This is a relative empirical-support score rather than a calibrated probability.

\subsubsection{Predictive Scenarios, Rank Gates, and Risk Fallback}
\label{app:scenario-sampling}

The implementation can draw correlated Gaussian-process samples for small candidate pools. Automatic correlated sampling is enabled only when the candidate pool contains at most 256 candidates. Since all reported experiments use $N=10{,}000$, they use independent marginal draws
\begin{equation}
f^{(m)}(x_j)
=
\mu(x_j)+\sigma_y(x_j)z_{m,j},
\qquad
z_{m,j}\overset{\mathrm{iid}}{\sim}\mathcal N(0,1),
\end{equation}
for $m=1,\ldots,M$ and $j=1,\ldots,N$. These samples preserve candidate-wise predictive means and variances but not cross-candidate posterior covariance.

For the greater-than targets used here,
\begin{equation}
T_m^{\mathrm{relax}}
=
\min\!\left\{
T,\;
Q_{1-\pi}
\bigl(\{f^{(m)}(x):x\in\mathcal C\}\bigr)
\right\},
\qquad
\pi=0.05,
\end{equation}
and
\begin{equation}
\psi_m(x)
=
\operatorname{sigm}
\left(
\frac{f^{(m)}(x)-T_m^{\mathrm{relax}}}{\tau}
\right).
\end{equation}
The effective soft-response temperature in all reported runs is $\tau=0.05$.

For a score vector $q(x)$ over a pool of size $N$, average ranks are used for ties, the best candidate receives rank percentile one, and the worst receives $1/N$. The gates are
\begin{equation}
g_q(x)
=
\left[
\epsilon_q+(1-\epsilon_q)R_q(x)
\right]^{\alpha_q},
\qquad
q\in\{h,v,s\},
\end{equation}
with
\begin{equation}
\epsilon_h=\epsilon_v=\epsilon_s=0.05,
\qquad
\alpha_h=1,\quad
\alpha_v=0.5,\quad
\alpha_s=0.5.
\end{equation}

For a partial batch $\mathcal S$ and candidate $x$, define the risk violation
\begin{equation}
V_R(x\mid\mathcal S)
=
\left[C(\mathcal S\cup\{x\})-R\right]_+.
\end{equation}
If a risk-feasible extension exists, ARC-SC maximizes the stage-specific desirability over the feasible extensions: $h(x)$ for POF anchoring and $\Delta_{\mathrm{SC}}(x\mid\mathcal S)$ for scenario completion. If none exists, it selects
\begin{equation}
x^\star
\in
\arg\max_{x\in\mathcal C\setminus\mathcal S}
\left\{
D(x\mid\mathcal S)
-
\lambda_R V_R(x\mid\mathcal S)
\right\},
\qquad
\lambda_R=10.
\label{eq:app-risk-fallback}
\end{equation}
The fallback ensures that a full batch can always be returned, but it means that $R$ is enforced whenever a feasible extension exists rather than as an unconditional terminal-batch guarantee.

\subsubsection{Hyperparameters Used in the Experiments}
\label{app:hyperparameters}

\begin{table}[t]
\centering
\caption{Algorithm and evaluation settings. Dataset-specific entries are shown explicitly.}
\label{tab:app-hyperparameters}
\begin{tabular}{lll}
\toprule
Quantity & Superconductivity & JARVIS CFID\\
\midrule
Initial observations $n_0$ & 50 & 50\\
Batch size $B$ / rounds & 5 / 5 & 5 / 5\\
POF anchors $K$ & 3 & 3\\
Candidate-pool size $N$ & 10,000 & 10,000\\
Predictive scenarios $M$ & 64 & 64\\
Target quantile & 0.85 & 0.75\\
Risk budget $R$ & 0.40 & 0.80\\
Uniform bounds padding $\rho_b$ & 0.05 & 0.05\\
Support neighbors $k$ & 5 & 5\\
Support-scale quantile & 0.75 & 0.75\\
Relaxation fraction $\pi$ & 0.05 & 0.05\\
Effective temperature $\tau$ & 0.05 & 0.05\\
Rank floors $(\epsilon_h,\epsilon_v,\epsilon_s)$ & $(.05,.05,.05)$ & $(.05,.05,.05)$\\
Rank exponents $(\alpha_h,\alpha_v,\alpha_s)$ & $(1,.5,.5)$ & $(1,.5,.5)$\\
Fallback coefficient $\lambda_R$ & 10 & 10\\
Correlated-sampling limit & 256 & 256\\
Seed IDs & 5--14 & 5--14\\
Main candidate pools & uniform & mixed; see Appendix~\ref{app:jarvis-trust-region}\\
Trust-region noise scale & -- & $0.05\times$ observed IQR\\
\bottomrule
\end{tabular}
\end{table}

The JARVIS risk budget $R=0.80$ is a dataset-specific operating setting held fixed across the reported JARVIS grids.

\subsection{Proof of Greedy-Completion Guarantee}
\label{app:theory}

Throughout this section, the observed dataset, finite candidate pool, predictive scenarios, and response matrix are conditioned upon and treated as fixed.

\subsubsection{Greedy-Approximation Derivation}
\label{app:greedy-proof}

Let $\mathcal A$ be a fixed anchor set and let $k=B-K\ge 1$. Over a fixed admissible completion pool $\mathcal F\subseteq\mathcal C\setminus\mathcal A$, define
\begin{equation}
G_{\mathcal A}(\mathcal S)
=A_{\mathrm{SC}}(\mathcal A\cup\mathcal S)
-A_{\mathrm{SC}}(\mathcal A).
\end{equation}
As established in Proposition~\ref{prop:residual-submodularity}, $G_{\mathcal A}$ is normalized, monotone, and submodular. Let $\mathcal S_i$ be the greedy completion after $i$ selections, with $\mathcal S_0=\varnothing$, and let
\begin{equation}
\mathcal S^\star\in
\arg\max_{\substack{\mathcal S\subseteq\mathcal F\\|\mathcal S|\le k}}
G_{\mathcal A}(\mathcal S)
\end{equation}
be an optimal completion.

At step $i$, monotonicity and submodularity imply
\begin{align}
G_{\mathcal A}(\mathcal S^\star)-G_{\mathcal A}(\mathcal S_{i-1})
&\le
G_{\mathcal A}(\mathcal S^\star\cup\mathcal S_{i-1})
-G_{\mathcal A}(\mathcal S_{i-1})\\
&\le
\sum_{x\in\mathcal S^\star\setminus\mathcal S_{i-1}}
\Delta_{\mathcal A}(x\mid\mathcal S_{i-1}),
\label{eq:app-gap-sum}
\end{align}
where $\Delta_{\mathcal A}$ is the marginal gain of the residual objective. Since $|\mathcal S^\star|\le k$, at least one candidate in the sum has marginal gain no smaller than
\begin{equation}
\frac{G_{\mathcal A}(\mathcal S^\star)
-G_{\mathcal A}(\mathcal S_{i-1})}{k}.
\end{equation}
The exact greedy candidate has at least this marginal gain, and hence
\begin{equation}
G_{\mathcal A}(\mathcal S_i)
-G_{\mathcal A}(\mathcal S_{i-1})
\ge
\frac{G_{\mathcal A}(\mathcal S^\star)
-G_{\mathcal A}(\mathcal S_{i-1})}{k}.
\label{eq:app-greedy-recursion}
\end{equation}
Define the remaining gap
\begin{equation}
E_i=G_{\mathcal A}(\mathcal S^\star)-G_{\mathcal A}(\mathcal S_i).
\end{equation}
Equation~\eqref{eq:app-greedy-recursion} gives
\begin{equation}
E_i\le\left(1-\frac{1}{k}\right)E_{i-1}.
\end{equation}
Recursing for $k$ steps and using $G_{\mathcal A}(\varnothing)=0$ yields
\begin{equation}
E_k\le
\left(1-\frac{1}{k}\right)^k
G_{\mathcal A}(\mathcal S^\star).
\end{equation}
Therefore,
\begin{align}
G_{\mathcal A}(\mathcal S_k)
\ge
\left[1-\left(1-\frac{1}{k}\right)^k\right]
G_{\mathcal A}(\mathcal S^\star)
\ge
\left(1-\frac{1}{e}\right)
G_{\mathcal A}(\mathcal S^\star),
\end{align}
which is the guarantee stated in Corollary~\ref{cor:greedy-completion}. This result compares greedy completion only with the optimal completion of the same fixed anchors over the same fixed admissible pool.

\subsection{Frozen-Oracle Benchmark Construction}
\label{app:oracle-construction}

The frozen-oracle protocol separates response-surface construction from closed-loop acquisition evaluation. Original dataset labels are used to select descriptors, train candidate regression models, and assess held-out quality. Once an oracle is selected, its predictions are frozen and used as the benchmark response for both existing reference rows and newly generated candidates. The closed-loop algorithms never alternate between original labels and oracle predictions.

\subsubsection{UCI Superconductivity Benchmark}

The source table contains 21,263 rows, 81 numerical composition descriptors, and the critical-temperature target. Rows containing missing numerical values are removed; no rows were removed in the downloaded table. An 80/20 split with random seed 42 is formed before descriptor selection.

Each descriptor is scored using three training-only statistics: absolute Spearman correlation with the target, mutual information, and Extra-Trees feature importance. Each statistic is converted to a percentile rank and the three ranks are averaged. The selector prioritizes a prespecified list of physically interpretable descriptors and excludes a candidate if its absolute Spearman correlation with any already selected descriptor exceeds 0.95. The selected five descriptors are
\begin{enumerate}
\item \texttt{wtd\_mean\_Valence},
\item \texttt{wtd\_std\_ThermalConductivity},
\item \texttt{wtd\_std\_atomic\_radius},
\item \texttt{wtd\_entropy\_atomic\_mass}, and
\item \texttt{wtd\_std\_fie}.
\end{enumerate}

Three oracle candidates are trained on the selected descriptors: a random forest with 800 trees, an Extra-Trees regressor with 800 trees, and histogram gradient boosting with 800 iterations, learning rate 0.04, and 31 leaf nodes. Models are ordered primarily by held-out Spearman correlation and secondarily by $R^2$. Table~\ref{tab:app-super-oracles} reports the held-out metrics. Extra Trees is selected as the frozen oracle.

\begin{table}[t]
\centering
\caption{Held-out superconductivity oracle performance.}
\label{tab:app-super-oracles}
\begin{tabular}{lrrrr}
\toprule
Model & $R^2$ & RMSE & MAE & Spearman $\rho$\\
\midrule
Extra Trees & 0.9128 & 10.0184 & 5.6564 & 0.9556\\
Random forest & 0.9095 & 10.2057 & 5.8667 & 0.9535\\
Hist. gradient boosting & 0.8932 & 11.0899 & 6.8611 & 0.9352\\
\bottomrule
\end{tabular}
\end{table}

The superconductivity discovery threshold is the $0.85$ quantile of the frozen-oracle predictions over the reference dataset, which equals $T=80.0$ in the reported benchmark.

\subsubsection{JARVIS CFID Benchmark}

The JARVIS CFID table contains 55,723 rows. Expanding the stored CFID vectors produces 1,543 candidate numerical descriptor columns. The original formation energy per atom is transformed as
\begin{equation}
y=-\mathrm{formation\_energy\_per\_atom},
\end{equation}
so that both benchmarks use a greater-than target. The data are divided into an 80/20 train/test split with random seed 42. Missing descriptor values are handled by median imputation inside the oracle pipelines.

Descriptor scores again average percentile ranks of absolute Spearman correlation, mutual information, and Extra-Trees importance. To control feature-selection cost, scoring uses at most 20,000 training rows. The 250 highest-scoring descriptors are considered sequentially, and candidates with absolute Spearman correlation greater than 0.98 with an already selected descriptor are excluded. The selected descriptors are
\begin{quote}\small
\texttt{cfid\_0048}, \texttt{cfid\_0124}, \texttt{cfid\_0240},
\texttt{cfid\_0168}, \texttt{cfid\_0075}, \texttt{cfid\_0384},
\texttt{cfid\_0405}, \texttt{cfid\_0324}, \texttt{cfid\_0236},
\texttt{cfid\_0189}, \texttt{cfid\_0345}, \texttt{cfid\_0134},
\texttt{cfid\_0361}, \texttt{cfid\_0505}, and \texttt{cfid\_0183}.
\end{quote}

The candidate oracles are a 700-tree Extra-Trees regressor, a 500-tree random forest, and histogram gradient boosting with 700 iterations, learning rate 0.04, and 31 leaf nodes. Each pipeline uses median imputation. Table~\ref{tab:app-jarvis-oracles} reports the held-out metrics. The random forest is selected as the frozen JARVIS oracle.

\begin{table}[t]
\centering
\caption{Held-out JARVIS oracle performance.}
\label{tab:app-jarvis-oracles}
\begin{tabular}{lrrrr}
\toprule
Model & $R^2$ & RMSE & MAE & Spearman $\rho$\\
\midrule
Random forest & 0.8375 & 0.4320 & 0.2518 & 0.9121\\
Extra Trees & 0.8311 & 0.4404 & 0.2531 & 0.9088\\
Hist. gradient boosting & 0.8239 & 0.4498 & 0.2828 & 0.9002\\
\bottomrule
\end{tabular}
\end{table}

\subsection{Synthetic Failure Mechanisms}
\label{app:failure-mechanisms}

The source datasets provide target properties but not controlled experimental-failure labels. We therefore apply a deterministic synthetic failure mask to the frozen-oracle response. All mechanisms are target independent: they depend on descriptor-space geometry and hidden feature interactions, while the frozen oracle separately determines the raw target value.

Let $x$ be robustly standardized using the component-wise median and interquartile range of the full reference feature matrix. If an interquartile range is numerically zero, the standard deviation is used, followed by a fallback scale of one. Denote the standardized vector by $\bar x\in\mathbb R^d$. For every experimental condition, three random unit vectors $w_1,w_2,w_3$ are generated from its deterministic seed, and
\begin{equation}
z_j=\bar x^\top w_j,\qquad j=1,2,3.
\end{equation}

\subsubsection{Off-Manifold Failure}

A ten-nearest-neighbor model is fitted to the robustly standardized reference data, and $d_{\mathrm{man}}(x)$ is the mean neighbor distance. Let $d_{0.8}$ be the $80$th percentile and $s_d$ the standard deviation of reference distances. The off-manifold risk is
\begin{equation}
r_{\mathrm{dist}}(x)=
\operatorname{sigm}\left(
\frac{d_{\mathrm{man}}(x)-d_{0.8}}{s_d}
\right).
\end{equation}

\subsubsection{Low-Validity-Probability Failure}

The smoothly varying nonlinear risk is
\begin{equation}
r_{\mathrm{nonlin}}(x)=
\operatorname{sigm}\left(
0.7\sin z_1+0.3\cos z_2+
0.15\frac{z_1z_2}{\sqrt d}
\right).
\end{equation}
This mechanism creates a hidden validity landscape that is not reducible to distance from the reference manifold.

\subsubsection{Feature-Interaction Failure}

The interaction mechanism uses
\begin{equation}
r_{\mathrm{int}}(x)=
\operatorname{sigm}\left(
0.70\frac{z_1z_2}{\sqrt d}
+0.20|z_3|
+0.10\max\{z_1-z_2,0\}
\right).
\end{equation}
It represents a hidden process window governed by combinations of descriptor directions rather than by a single coordinate.

\subsubsection{Hybrid Failure}

The hybrid risk is the weighted average
\begin{equation}
r_{\mathrm{hyb}}(x)=
0.60r_{\mathrm{dist}}(x)
+0.25r_{\mathrm{nonlin}}(x)
+0.15r_{\mathrm{int}}(x).
\end{equation}
The weights sum to one.

\subsubsection{Calibration to a Nominal Invalid Rate}

For any selected risk function $r(x)$ and nominal invalid rate $\rho\in\{0.2,0.3,0.4\}$, the deterministic threshold is
\begin{equation}
\theta_\rho=Q_{1-\rho}
\bigl(\{r(x):x\in\mathcal X_{\mathrm{ref}}\}\bigr).
\end{equation}
A queried candidate is invalid when $r(x)\ge\theta_\rho$. Its raw frozen-oracle response is retained for diagnostics, but the observed target is set to missing and the validity label is set to zero. Calibration is performed separately for every dataset, mechanism, nominal invalid rate, and seed. Because empirical quantiles can contain ties, the achieved reference invalid rate is approximately, rather than necessarily exactly, $\rho$.

The four synthetic failure mechanisms represent qualitatively different sources of experimental invalidity. Low-validity failure creates nonlinear regions of design space in which experiments are intrinsically less likely to succeed. Feature-interaction failure represents hidden process constraints that arise from combinations of otherwise individually plausible descriptors. Off-manifold failure penalizes extrapolative candidates that lie far from the empirical distribution of known designs. Finally, the hybrid mechanism combines these effects to represent experimental environments in which invalidity may arise from multiple causes. Together, these mechanisms test whether an acquisition strategy remains robust to local feasibility structure, variable interactions, extrapolation risk, and their combination.

\subsection{Experimental Protocol and Statistical Analysis}
\label{app:experimental-details}

\subsubsection{Holdout-Seed Evaluation Grid}

The results reported in the main paper use the seed IDs $5,\ldots,14$; Each dataset combines four failure mechanisms, three nominal invalid rates, three initialization quantiles, and ten seed IDs, $4\times 3\times 3\times 10=360$ paired evaluation cases. Each run begins with $n_0=50$ reference designs sampled without replacement from rows whose frozen-oracle response lies below the specified initialization quantile ($0.65$, $0.75$, or $0.85$), followed by five closed-loop rounds of batch size five.

For the shared uniform-pool protocol, ARC-SC and \textsc{POF} receive the same initial data, frozen oracle, failure mechanism, target threshold, candidate pool, and model seeds. The candidate pool at round $t$ is generated once from the common condition seed and then passed to both methods. For the JARVIS trust-region mechanisms, the initial condition and failure environment remain paired, but candidate pools are generated from each method's own observed trajectory as described in Appendix~\ref{app:jarvis-trust-region}.

The first-hit metric is computed only over newly queried rounds 1--5. Initial target-level observations, when present, are recorded diagnostically but do not count as a closed-loop first hit. This matters for JARVIS when the initialization quantile is $0.85$ while the target quantile is $0.75$.

\subsubsection{Primary and Secondary Endpoints}

Let $H=5$ denote the closed-loop horizon. The censored first-hit round is
\begin{equation}
\widetilde{\tau}_{\mathrm{hit}}
=
\begin{cases}
\tau_{\mathrm{hit}}, & \text{if a newly queried hit is found},\\
H+1=6, & \text{otherwise}.
\end{cases}
\end{equation}
The paired first-hit effect is
\begin{equation}
\Delta\widetilde{\tau}_{\mathrm{hit}}
=
\widetilde{\tau}_{\mathrm{hit}}^{\mathrm{ARC\text{-}SC}}
-
\widetilde{\tau}_{\mathrm{hit}}^{\mathrm{POF}},
\end{equation}
so negative values favor ARC-SC. A paired run is an earlier-hit win, tie, or later-hit loss according to whether this difference is negative, zero, or positive.

Secondary effects are
\begin{equation}
\Delta N_{\mathrm{hit}}
=
N_{\mathrm{hit}}^{\mathrm{ARC\text{-}SC}}
-
N_{\mathrm{hit}}^{\mathrm{POF}},
\end{equation}
and
\begin{equation}
\Delta r_{\mathrm{invalid}}
=
r_{\mathrm{invalid}}^{\mathrm{ARC\text{-}SC}}
-
r_{\mathrm{invalid}}^{\mathrm{POF}},
\end{equation}
where positive $\Delta N_{\mathrm{hit}}$ and negative $\Delta r_{\mathrm{invalid}}$ favor ARC-SC.

\subsubsection{Exact Sign Test and Hierarchical Bootstrap}

For the sign test, ties are removed and the number of ARC-SC wins among the discordant pairs is tested against a $\operatorname{Binomial}(W+L,0.5)$ null model. The main paper reports a two-sided exact $p$-value and an exact $95\%$ confidence interval for the conditional win fraction $W/(W+L)$.

To quantify effect-size uncertainty while respecting the factorial design, we use a two-level hierarchical bootstrap with 10,000 replicates. The 36 experimental-condition cells are defined by failure mechanism, nominal invalid rate, and initialization quantile. Each bootstrap replicate (i) samples 36 condition cells with replacement and then (ii) resamples the ten paired seed-level observations with replacement within each sampled cell. The statistic is computed within each resampled cell and then averaged across the sampled cells. This procedure is used for the paired first-hit, cumulative-hit, invalid-rate, and discovery-by-round intervals. It therefore reflects variation over both the prespecified condition grid and stochastic repeats rather than treating all 360 rows as homogeneous IID observations.

\end{document}

%% file: math_commands.tex
\usepackage{amsmath,amsfonts,bm}

\def\eqref#1{equation~\ref{#1}}

\def\1{\bm{1}}

\DeclareMathAlphabet{\mathsfit}{\encodingdefault}{\sfdefault}{m}{sl}
\SetMathAlphabet{\mathsfit}{bold}{\encodingdefault}{\sfdefault}{bx}{n}

%% file: references.bib
@article{hickman2025anubis,
  title={Anubis: Bayesian optimization with unknown feasibility constraints for scientific experimentation},
  author={Hickman, Riley J and Tom, Gary and Zou, Yunheng and Aldeghi, Matteo and Aspuru-Guzik, Al{\'a}n},
  journal={Digital Discovery},
  volume={4},
  number={8},
  pages={2104--2122},
  year={2025},
  publisher={The Royal Society of Chemistry}
}

@article{nemhauser1978analysis,
  title={An analysis of approximations for maximizing submodular set functions—I},
  author={Nemhauser, George L and Wolsey, Laurence A and Fisher, Marshall L},
  journal={Mathematical programming},
  volume={14},
  number={1},
  pages={265--294},
  year={1978},
  publisher={Springer}
}

@article{azimi2010batch,
  title={Batch Bayesian optimization via simulation matching},
  author={Azimi, Javad and Fern, Alan and Fern, Xiaoli},
  journal={Advances in neural information processing systems},
  volume={23},
  year={2010}
}

@article{shah2015parallel,
  title={Parallel predictive entropy search for batch global optimization of expensive objective functions},
  author={Shah, Amar and Ghahramani, Zoubin},
  journal={Advances in neural information processing systems},
  volume={28},
  year={2015}
}

@inproceedings{gonzalez2016batch,
  title={Batch Bayesian optimization via local penalization},
  author={Gonz{\'a}lez, Javier and Dai, Zhenwen and Hennig, Philipp and Lawrence, Neil},
  booktitle={Artificial intelligence and statistics},
  pages={648--657},
  year={2016},
  organization={PMLR}
}

@article{kathuria2016batched,
  title={Batched gaussian process bandit optimization via determinantal point processes},
  author={Kathuria, Tarun and Deshpande, Amit and Kohli, Pushmeet},
  journal={Advances in neural information processing systems},
  volume={29},
  year={2016}
}

@article{wu2016parallel,
  title={The parallel knowledge gradient method for batch Bayesian optimization},
  author={Wu, Jian and Frazier, Peter},
  journal={Advances in neural information processing systems},
  volume={29},
  year={2016}
}

@inproceedings{gardner2014bayesian,
  title={Bayesian optimization with inequality constraints.},
  author={Gardner, Jacob R and Kusner, Matt J and Xu, Zhixiang Eddie and Weinberger, Kilian Q and Cunningham, John P},
  booktitle={ICML},
  volume={2014},
  pages={937--945},
  year={2014}
}

@inproceedings{hernandez2015predictive,
  title={Predictive entropy search for Bayesian optimization with unknown constraints},
  author={Hern{\'a}ndez-Lobato, Jos{\'e} Miguel and Gelbart, Michael and Hoffman, Matthew and Adams, Ryan and Ghahramani, Zoubin},
  booktitle={International conference on machine learning},
  pages={1699--1707},
  year={2015},
  organization={PMLR}
}

@inproceedings{tian2024boundary,
  title={Boundary exploration for Bayesian optimization with unknown physical constraints},
  author={Tian, Yunsheng and Zuniga, Ane and Zhang, Xinwei and D{\"u}rholt, Johannes P and Das, Payel and Chen, Jie and Matusik, Wojciech and Lukovi{\'c}, Mina Konakovi{\'c}},
  booktitle={Proceedings of the 41st International Conference on Machine Learning},
  pages={48295--48320},
  year={2024}
}

@inproceedings{bergmann2020safe,
  title={Safe Bayesian optimization under unknown constraints},
  author={Bergmann, Daniel and Graichen, Knut},
  booktitle={2020 59th IEEE Conference on Decision and Control (CDC)},
  pages={3592--3597},
  year={2020},
  organization={IEEE}
}

@article{wakabayashi2022bayesian,
  title={Bayesian optimization with experimental failure for high-throughput materials growth},
  author={Wakabayashi, Yuki K and Otsuka, Takuma and Krockenberger, Yoshiharu and Sawada, Hiroshi and Taniyasu, Yoshitaka and Yamamoto, Hideki},
  journal={npj Computational Materials},
  volume={8},
  number={1},
  pages={180},
  year={2022},
  publisher={Nature Publishing Group UK London}
}

@inproceedings{iwazaki2025no,
  title={No-Regret Bayesian Optimization with Stochastic Observation Failures},
  author={Iwazaki, Shogo and Tanabe, Tomohiko and Irie, Mitsuru and Takeno, Shion and Matsui, Kota and Inatsu, Yu},
  booktitle={International Conference on Artificial Intelligence and Statistics},
  pages={415--423},
  year={2025},
  organization={PMLR}
}

@inproceedings{garnett2012bayesian,
  title={Bayesian optimal active search and surveying},
  author={Garnett, Roman and Krishnamurthy, Yamuna and Xiong, Xuehan and Schneider, Jeff and Mann, Richard},
  booktitle={Proceedings of the 29th International Coference on International Conference on Machine Learning},
  pages={843--850},
  year={2012}
}

@inproceedings{jiang2017efficient,
  title={Efficient nonmyopic active search},
  author={Jiang, Shali and Malkomes, Gustavo and Converse, Geoff and Shofner, Alyssa and Moseley, Benjamin and Garnett, Roman},
  booktitle={International Conference on Machine Learning},
  pages={1714--1723},
  year={2017},
  organization={PMLR}
}

@inproceedings{chen2013near,
  title={Near-optimal batch mode active learning and adaptive submodular optimization},
  author={Chen, Yuxin and Krause, Andreas},
  booktitle={International Conference on Machine Learning},
  pages={160--168},
  year={2013},
  organization={PMLR}
}

@article{lookman2019active,
  title={Active learning in materials science with emphasis on adaptive sampling using uncertainties for targeted design},
  author={Lookman, Turab and Balachandran, Prasanna V and Xue, Dezhen and Yuan, Ruihao},
  journal={npj Computational Materials},
  volume={5},
  number={1},
  pages={21},
  year={2019},
  publisher={Nature Publishing Group UK London}
}

@article{xue2016accelerated,
  title={Accelerated search for materials with targeted properties by adaptive design},
  author={Xue, Dezhen and Balachandran, Prasanna V and Hogden, John and Theiler, James and Xue, Deqing and Lookman, Turab},
  journal={Nature communications},
  volume={7},
  number={1},
  pages={11241},
  year={2016},
  publisher={Nature Publishing Group}
}

@article{kusne2020fly,
  title={On-the-fly closed-loop materials discovery via Bayesian active learning},
  author={Kusne, A Gilad and Yu, Heshan and Wu, Changming and Zhang, Huairuo and Hattrick-Simpers, Jason and DeCost, Brian and Sarker, Suchismita and Oses, Corey and Toher, Cormac and Curtarolo, Stefano and others},
  journal={Nature communications},
  volume={11},
  number={1},
  pages={5966},
  year={2020},
  publisher={Nature Publishing Group UK London}
}

@misc{superconductivty_data_464,
  author       = {Hamidieh, Kam},
  title        = {{Superconductivty Data}},
  year         = {2018},
  howpublished = {UCI Machine Learning Repository},
  note         = {{DOI}: https://doi.org/10.24432/C53P47}
}

@article{choudhary2020joint,
  title={The joint automated repository for various integrated simulations (JARVIS) for data-driven materials design},
  author={Choudhary, Kamal and Garrity, Kevin F and Reid, Andrew CE and DeCost, Brian and Biacchi, Adam J and Hight Walker, Angela R and Trautt, Zachary and Hattrick-Simpers, Jason and Kusne, A Gilad and Centrone, Andrea and others},
  journal={npj computational materials},
  volume={6},
  number={1},
  pages={173},
  year={2020},
  publisher={Nature Publishing Group UK London}
}

@article{wu2026matformbench,
  title={MatFormBench: A Benchmarking Evaluation Framework for Target-Driven Materials Formulation},
  author={Wu, Linhan and Wang, Chenxi and Yang, Chuhan and Yang, Zhengwei and Liu, Yuyang},
  journal={arXiv preprint arXiv:2605.26741},
  year={2026}
}

@article{eriksson2019scalable,
  title={Scalable global optimization via local Bayesian optimization},
  author={Eriksson, David and Pearce, Michael and Gardner, Jacob and Turner, Ryan D and Poloczek, Matthias},
  journal={Advances in neural information processing systems},
  volume={32},
  year={2019}
}

@article{papenmeier2022increasing,
  title={Increasing the scope as you learn: Adaptive bayesian optimization in nested subspaces},
  author={Papenmeier, Leonard and Nardi, Luigi and Poloczek, Matthias},
  journal={Advances in Neural Information Processing Systems},
  volume={35},
  pages={11586--11601},
  year={2022}
}

@article{volk2024performance,
  title={Performance metrics to unleash the power of self-driving labs in chemistry and materials science},
  author={Volk, Amanda A and Abolhasani, Milad},
  journal={Nature communications},
  volume={15},
  number={1},
  pages={1378},
  year={2024},
  publisher={Nature Publishing Group UK London}
}

@article{tom2024self,
  title={Self-driving laboratories for chemistry and materials science},
  author={Tom, Gary and Schmid, Stefan P and Baird, Sterling G and Cao, Yang and Darvish, Kourosh and Hao, Han and Lo, Stanley and Pablo-Garc{\'\i}a, Sergio and Rajaonson, Ella M and Skreta, Marta and others},
  journal={Chemical Reviews},
  volume={124},
  number={16},
  pages={9633},
  year={2024}
}

@article{wang2024bayesian,
  title={Bayesian Occam’s Razor to Optimize Metamodeling for Complex Biological Systems},
  author={Wang, Chenxi and Zhao, Jihui and Zheng, Jingjing and Raveh, Barak and He, Xuming and Sun, Liping},
  journal={bioRxiv},
  pages={2024--05},
  year={2024},
  publisher={Cold Spring Harbor Laboratory}
}
